\documentclass{amsart}
\usepackage{amsfonts, amssymb, amscd, amsmath,amsrefs}

\newtheorem{theorem}[equation]{Theorem}
\newtheorem{corollary}[equation]{Corollary}
\newtheorem{proposition}[equation]{Proposition}
 
\theoremstyle{definition}
\newtheorem{definition}[equation]{Definition}
\newtheorem{example}[equation]{Example} 
\theoremstyle{remark}
\newtheorem{remark}[equation]{Remark} 

\numberwithin{equation}{section}

\newcommand{\cE}{{\mathcal E}}
\newcommand{\cF}{{\mathcal F}}
\newcommand{\cK}{{\mathcal K}}
\newcommand{\cL}{{\mathcal L}}
\newcommand{\cO}{{\mathcal O}}
\newcommand{\cP}{{\mathcal P}}
\newcommand{\cN}{{\mathcal N}}
\newcommand{\id}{{\mathrm {id}}}
\newcommand{\pr}{{\mathrm {pr}}}
\newcommand{\Osc}{{\mathrm {Osc}}}
\newcommand{\Proj}{{\mathrm {Proj}}}
\newcommand{\Grass}{{\mathrm {Grass}}}
\newcommand{\IP}{{\mathbb P}}
\newcommand{\C}{\mathbb C}

\DeclareMathOperator{\Ker}{Ker}
\DeclareMathOperator{\Image}{Im}
\DeclareMathOperator{\rank}{rank}

\def\P#1{{\cP}^{#1}_X({\cL})}
 
\begin{document}

\title[On fundamental forms and
osculating bundles]{On fundamental forms and osculating bundles}
\author{Raquel Mallavibarrena} 
\address{Departamento de Algebra, Facultad de Ciencias Matemáticas,
Plaza de las Ciencias, 3 - Universidad Complutense de Madrid,
28040 Madrid, Spain}
\email{rmallavi@mat.ucm.es}
\author{Ragni Piene} 
\address{Department of Mathematics, University
of Oslo, P.O. Box 1053 Blindern, NO-0316 Oslo, Norway}
\email{ragnip@math.uio.no}


\maketitle

\hfill{\emph{To the memory of Gianni Sacchiero}}

\begin{abstract}
We define higher order fundamental forms and osculating spaces of projective algebraic varieties, using sheaves of principal parts. We show that the $m$th fundamental form can be viewed as the differential of the $(m-1)$th Gauss map, and explain why the vanishing of the $m$th fundamental form implies that the variety is contained in a general $(m-1)$th osculating space. Pointwise, the fundamental forms give linear systems on the projectivized tangent spaces. We show that, at each point, the Jacobian of the $m$th fundamental form is contained in the $(m-1)$th fundamental form. In the case of ruled varieties, we describe these linear systems. We discuss conditions for a surface to be ruled, in terms of the second fundamental form and the Fubini cubic.
\end{abstract}

\section{Introduction}
In classical differential geometry the \emph{second fundamental form} of a surface in $\mathbb{R}^3$ at a smooth point is a quadratic form on the tangent space to the surface at that point. The starting point of this paper is the work by Griffiths and Harris published in 1979 \cite{GH}. Using Darboux frames, they defined this quadratic form pointwise for a complex analytic projective variety and showed that it could be viewed as the fiber of a map from the second symmetric product of the tangent bundle to the normal bundle \cite{GH}*{(1.18), p.~366}. They also defined higher fundamental forms pointwise, using Darboux frames, and gave a similar description of the corresponding maps of bundles. Our reading of their paper led us, more than thirty years ago, to 
define, in a purely algebraic
way and without using frames, the higher fundamental forms of a quasi-projective variety. This definition was not published at that time, but appeared in lectures and in the Master thesis of Tegnander \cite{Te}.
A similar definition was recently given by Ein and Niu in their paper \cite{EN}, which  made us revisit our old notes and expand them into the present paper, where we
 give various interpretations and properties of these fundamental forms and the linear systems
associated with them. Several of our results can also be found, or have analogs, in papers by other authors, such as \cite{GH}, \cite{S}, \cite{L}, \cite{DeDiI}.

The paper is organized as follows. In the next section we define fundamental forms of an algebraic variety $X$ in projective space, using sheaves of principal parts
that define the osculating spaces of the variety.
We relate our definition to the definition by Altman and Kleiman \cite{AK1} of the second fundamental form of a subsheaf and prove Theorem \ref{fact}, which generalizes an observation by Perkinson \cite{P}. 
We define higher order Gauss maps, and show that Proposition \ref{fund} and Theorem \ref{fact} imply that
 the $m$th fundamental form is equal to the differential of the $(m-1)$th Gauss map. We also explain why the vanishing of the $m$th fundamental form implies that $X$ is contained in its $(m-1)$th osculating space at a (general) point.

In the third section we consider the interpretation of fundamental forms as linear systems on the projectivized tangent spaces. We prove in Theorem \ref{jac} that the Jacobian of the $m$th fundamental form is contained in the $(m-1)$th fundamental form. We illustrate our results by three examples; these are non-ruled surfaces in $\IP^5$ such that the second order osculating spaces have dimension $4$ (instead of the expected dimension $5$). The second fundamental forms are pencils of quadrics in $\IP^1$; in one case, these pencils have a base point, in the two other cases, they do not. 

In the fourth section we study and describe the fundamental forms of projective ruled varieties $\pi \colon X=\IP(\mathcal E)\to Y$. 
We use a result of Landsberg \cite{LandsbergLinear} to give a condition for a \emph{surface} to be ruled, in terms of the second fundamental form and the Fubini cubic discussed in \cite{GH}.
We can view ruled varieties as varieties in a Grassmann variety, and we show that the bundles of principal parts of $\mathcal E$ on $Y$  are equal to the pushdowns of the bundles of principal parts of $\mathcal O_X(1)$ on $X$.

This work grew out of old notes by the authors. A part of these notes were based on  writings by the second author in 1988--89, while she was a Science 
Scholar at the Bunting Institute  of Radcliffe College,\footnote{Now the Radcliffe Institute for Advanced Study at Harvard University.} and on her  lecture  ``Espaces osculateurs, formes fondamentales et multiplicit\'es des discriminants'' given at \'Ecole Normale Sup\'erieure in Paris on November 26, 1992.
In  the 1992 Master thesis of Cathrine Tegnander this definition of fundamental forms is used \cite{Te}*{4.1}, and some of our surface examples are taken from her thesis.

\section{Fundamental forms}

Let $k$ be an algebraically closed field and $V$ a $k$-vector 
space of dimension $N+1$.  Suppose $X$ is a smooth (but not necessarily proper), irreducible
$k$-scheme, of dimension $r$, and that $f \colon X\to \IP(V)$ is a morphism which
is birational onto its image.

Set ${\cL} := f^{*}{\cO}_{\IP(V)}(1)$, and let ${\cP}^m_X({\cL})$ denote 
the sheaf of principal
parts of order $m$ of $\cL$, for $m\geq 0$ (see \cite{EGA}*{16.7}, \cite{numchar}*{§ 6, pp.~492--494}). Since $X$ is smooth, of 
dimension $r$,
${\cP}^m_X({\cL})$ is locally free, with rank $\binom {r+m}m$, and 
there are exact sequences,
for $m\geq 1$, 
\[0\to S^m\Omega ^1_X\otimes {\cL}\to \P m \to
\P {m-1}\to 0.\]
Moreover, for each $m$, there is a natural map 
\[H^0(X,\cL)_{X}\to {\cP}^m_X({\cL}),\]
and we denote by 
\[a^m \colon V_X\to {\cP}^m_X({\cL})\]
the map obtained by composing this map with the homomorphism
\[V_{X}=H^0(\IP(V),\cO_{\IP(V)}(1))_{X}\to H^0(X,\cL)_{X}.\]
The maps $a^m$ are locally just Taylor series expansion up to order
$m$ of the coordinate functions on $X$, with the variables being local
coordinates on $X$.  They are compatible with the
surjections 
\[{\cP}^m_X({\cL})\to {\cP}^{m-i}_X({\cL}).\]

Now we set $\cK_m=\Ker (a^m)$, and consider the maps $\phi_m$ defined by
the commutative diagrams
\[
\begin{CD}
    0 @>>> \cK_{m-1} @>>> V_X @>a^{m-1}>> {\P {m-1}} \\
@. @V{\phi _m}VV @V{a^m}VV @| @. \\
0 @>>> {S^m\Omega ^1_X\otimes {\cL}} @>>> {\P m} 
@>>> {\P {m-1}} @>>> 0 .\\
\end{CD}
\]
It follows that $\cK_m\subseteq \cK_{m-1}$ and that 
$\Ker (\phi_m) = \cK_m$.  
\medskip

\begin{definition} The $m$th fundamental form of $X$ (with respect to $f$)
is the injective homomorphism induced by $\phi_m$, 
\[\Phi_m \colon \cK_{m-1}/\cK_m \to
S^m\Omega ^1_X\otimes \cL .\]
\end{definition}

Let us first briefly compare this definition with the ``classical'' 
fundamental forms, as defined locally (see e.g. \cite{GH}*{(1.18), p.~366
and (1.46), p.~373} for $m=2,3$).  For this, we may assume that $f$ is an
embedding.  Let $T_X:=(\Omega^1_X)^\vee$ denote the tangent bundle
to $X$.  If $m$!  is invertible in $k$, the natural map
\[S^mT_X \to (S^m\Omega^1_X)^\vee\] 
is an isomorphism (see  \cite{AK2}*{Lemma
(2.13), p.~21}, \cite{FH}*{B.~3, p.476}, and \cite{M}*{p.~248}).  Hence we obtain a
map
\[S^mT_X\to \cK^\vee_{m-1}/\cK^\vee_m\otimes {\cL}\] 
by composing with $\Phi_m^\vee\otimes\id_{\cL}$.  

For $m=2$ we get
 \[\phi_2 \colon \cK_1 =\cN_{X/\IP(V)}\otimes \cL \to S^2\Omega ^1_X\otimes \cL ,\]
where $\cN_{X/\IP(V)}$ is the conormal sheaf of $X$ in $\IP(V)$, and hence a map
$S^2T_X \to \cN_{X/\IP(V)}^\vee$, whose fibers  are the classical second fundamental forms.
We shall consider the linear systems induced by the fundamental forms in Section \ref{geo}.
\medskip

Altman and Kleiman \cite{AK1}*{I.3, p.~10} gave a general
definition of the second fundamental form of a subsheaf of a
quasi-coherent sheaf on a scheme.  We shall now recall their
construction, in our situation.

Suppose $\cF$ is a coherent sheaf on $X$ and that 
\[\alpha \colon V_X\to \cF\] is a homomorphism. Set ${\cE}:=\Ker (\alpha)$, $Z:=X\times
\IP(V)$, $Y:=\IP(V_{X}/{\cE})\subseteq Z$, and let $\pr_1 \colon Z\to X$,
$\pr_2\colon Z\to \IP(V)$, and $p \colon Y\to X$ denote the projection morphisms. By 
\cite{AK2}*{Lemma (2.6), p.~17}, 
there is a natural
homomorphism 
\[p^*{\cE}\otimes {\cO}_Y(-1)\to {\cN_{Y/Z}},\] 
where
${\cN_{Y/Z}}$ denotes the conormal sheaf of $Y$ in $Z$, which is an
isomorphism if  ${\cF}$ is
locally free. Now we compose this homomorphism with the homomorphisms
\[{\cN_{Y/Z}}\to
{\Omega ^1_Z}|_Y=(\pr_1^*\Omega ^1_X\oplus \pr_2^*\Omega
^1_{\IP(V)})|_Y\to
\pr_1^* \Omega^1_X|_Y=p^*\Omega^1_X .\] 
Thus we obtain a homomorphism
\[p^* {\cE}\otimes {\cO}_Y(-1)\to p^*\Omega^1_X,\]
 hence also 
\[p^*{\cE}\to
p^*\Omega^1_X\otimes {\cO}_Y(1),\] 
and, by adjunction and the 
projection formula (since
$\Omega^1_X$ is locally free), 
\[{\cE}\to 
p_*(p^*\Omega^1_X\otimes {\cO}_Y(1))\cong
\Omega^1_X\otimes p_*{\cO}_Y(1).\] Let $U\subseteq X$ be the open 
dense subset where
$\alpha$ has constant rank.
On $U$, $p_* {\cO}_Y(1)\cong {\rm { Im }}
(\alpha)=V_{X}/{\cE}$, so, by composing
with the inclusion $V_{X}/{\cE} \to {\cF}$, we finally obtain 
a homomorphism, defined
on $U$,
and denoted $F(\alpha)$, 
\[F(\alpha) \colon \cE|_{U}\to \Omega^1_U\otimes \cF|_{U}.\]

\begin{definition} \cite{AK1}*{I.3, p.~10}.  The second
fundamental form of $\cE:=\Ker (\alpha)$ in $V_X$ is the
homomorphism 
\[F(\alpha) \colon \cE|_U\to (\Omega^1_X\otimes \cF)|_U\]
 constructed
above. 
\end{definition}

By \cite{AK1}*{Thm.  (3.1), p.~11} (see also \cite{F}*{B.5.8, p.~435}), the second
fundamental form of the kernel of the surjection $a^0 \colon  V_{\IP(V)}\to \cO_{\IP(V)}(1)$
identifies this kernel with the sheaf $\Omega_{\IP(V)}^1(1)$.  Via this
identification, we get $\cK_0= f^*\Omega^1_{\IP(V)} \otimes {\cL}$ and
$\phi_1=df\otimes \id_{\cL}$.  Moreover, if $U\subseteq X$ denotes the
open such that $f|_U$ is an embedding, then on $U$,
$\cK_1={\cN}_{f(X)/\IP(V)}|_U \otimes {\cL}$ and $\cK_0/\cK_1 \cong
\Omega^1_X \otimes {\cL}$, where ${\cN}_{f(X)/\IP(V)}$ denotes the conormal
sheaf of $f(X)$ in $\IP(V)$.  Hence $\Phi_1|_U = \id$ is trivial.
\medskip

As remarked by Perkinson \cite{P}*{Remark 2.4, p.~3183}, in this case the map
$\phi_1$, induced by the Taylor map $a^1 \colon V_{\IP(V)}\to \cP^1_{\IP(V)}(1)$, is equal to
$-F(a^0)$.  We shall generalize this in Theorem \ref{fact}.

\begin{proposition} \label{fund}
Assume ${\cF}$ is locally free, with rank $s+1$, and that 
\[\alpha \colon V_X\to \cF\]
 is surjective.  Set $\cE := \Ker
(\alpha)$, and $G:=\Grass_{s+1}(V)$, and let $\psi \colon X\to G$ denote the
morphism corresponding to $\alpha$.  Then the second fundamental form
of $\cE$, 
\[F(\alpha) \colon \cE \to \Omega^1_X\otimes \cF,\]
induces
\[d\psi \colon \psi^* \Omega^1_G\to \Omega^1_X\]
 via the isomorphism
$\psi^* \Omega^1_G\cong {\cE}\otimes \cF^\vee$.
\end{proposition} 

\begin{proof} See \cite{AK1}*{I.3, p.~10} and \cite{F}*{B.5.8, p.~435}.  
    \end{proof}

We shall now show how the $m$th fundamental form of $X$ (with respect
to the map $f \colon X\to \IP(V)$)
is related to the second fundamental form of the kernel $\cK_{m-1}$ of 
the homomorphism
\[a^{m-1} \colon V_X\to \P {m-1}.\]
We have the following result.

\begin{theorem}\label{fact} The second fundamental form of $\cK_{m-1}$,
\[F(a^{m-1}) \colon  \cK_{m-1}\to \Omega^1_X\otimes \P {m-1},\] 
factors through the inclusion 
\[\Omega^1_X\otimes
S^{m-1}\Omega^1_X\otimes {\cL}\hookrightarrow \Omega^1_X\otimes \P 
{m-1},\] 
and the induced homomorphism 
\[\overline \phi_m \colon \cK_{m-1}\to S^m\Omega^1_X\otimes {\cL}\] satisfies 
\[\overline \phi_m=-m\phi_m,\]
where $\phi_m$ is the map used to define the $m$th fundamental form of $X$.
\end{theorem}

\begin{proof} In order to show that $F(a^{m-1})$ factors as stated, we must show that
the image of $F(a^{m-1})(\cK_{m-1})$
is zero under the homomorphism 
\[\Omega^1_X\otimes \P {m-1}\to \Omega^1_X\otimes \P {m-2}.\]
This can be
checked locally around a point $x\in X$: Let $u_1$,...,$u_r$ be local 
parameters for $X$ at $x$
(i.e., generators for the maximal ideal ${\mathfrak m}$ of the 
local ring of $X$ at 
$x$), and let $x_0,\ldots,x_N\in
{\cO}_{X,x}\cong {\cO}_X(1)_x$ be the images of a basis $X_{0},\ldots,X_{N}$ for
$V=H^0(\IP(V),{\cO}_{\IP(V)}(1))$, so that the $x_j$ are functions of $u_1,\dots,u_r$.  A set of generators for the free
${\cO}_{X,x}-$module 
\[{\P {m-1}}_x \cong {\cP}^{m-1}_{X,x} \cong
{\cO}_{X,x}\otimes {\cO}_{X,x}/{\mathfrak m}^m_x\]
is then $\{
du^I\}_{|I|\leq {m-1}}$, where $I=(i_1,...,i_r)$ and
$du^I=(du_1)^{i_1}\cdot \cdot \cdot (du_r)^{i_r}$.  
Set $\partial_i:=\partial/\partial u_i$ and $\partial^I:=\partial_{i_1}^{i_1}\cdots \partial_{i_r}^{i_r}$.
The map $a^{m-1}_x$ is given by 
\[a^{m-1}_x(1\otimes X_j)=\sum_{|I|\leq{m-1}} D_I x_j du^I\]
 for $j=0,\ldots,N$, where $D_I:  =\frac{1}{i_1!\cdots i_r!}\partial^I$ is the Hasse
differential operator ``dual'' to $du^I$ given by the coefficients in
Taylor series expansions. In particular, $D_I(du^J)=\delta_{IJ}$ (the Kronecker delta).  (See \cite{EGA}*{16.11.2}, and
\cite{M}*{p.~248}.)

 From \cite{AK1}*{I.3, p.~10} it follows that the map $F(a^{m-1})_x$ is given 
by sending an element $g \in
(\cK_{m-1})_x$, 
\[g=\sum_{j=0}^N {g_j\otimes X_j}\in 
{\cO}_{X,x}\otimes V\]
to 
\[\sum_{j=0}^N
{d(g_j)\otimes a_x^{m-1}(1\otimes X_j)}\in \Omega^1_{X,x}\otimes {\P {m-1}}_x\]
(here $d=d_{x} \colon {\cO}_{X,x}\to \Omega^1_{X,x}\cong \oplus_{k=1}^r{\cO}_{X,x}du_k$, so that 
$dg=\sum_{k=1}^r {\partial}_k(g)
du_k$, where ${\partial}_k=D_{(0,...,1,...,0)}$ is differentiation 
with respect to $u_k$).

We need to show that the image of $g$ in 
$\Omega^1_{X,x}\otimes {\P {m-2}}_x$ is zero.
But this image is equal to 
\begin{align*}
\sum_{j=0}^N d(g_j)\otimes a_x^{m-2}(1\otimes 
x_j)
&=\sum_{j=0}^N \sum_{k=1}^r {\partial}_k g_j du_k\otimes 
\sum_{|I|\leq {m-2}} D_I x_j du^I \\
    &=\sum_{k=1}^r du_k\otimes \sum_{j=0}^N \partial_k g_j \sum_{|I|\leq 
{m-2}} D_I x_j du^I.
\end{align*}
Since $g\in (\cK_{m-1})_x=\Ker (a_x^{m-1})$, we have
\[0=a_x^{m-1}(\sum_{j=0}^N g_j\otimes x_j)=\sum_{j=0}^N g_j
\sum_{|I|\leq {m-1}} D_I x_j du^I.\]
 Since the $du^I$'s generate the
free ${\cO}_{X,x}$-module ${\P {m-1}}_x$, this implies that
\begin{equation}\label{(*)}
\sum_{j=0}^N g_j D_I x_j =0 \hbox{ \rm for each $I$ with } |I|\leq {m-1}. 
\end{equation} 
Applying the
differential operators $\partial_k$ to these equations, we get, for $|I|\le m-1$,
\[0=\partial_k(\sum_{j=0}^N g_j D_I
x_j)=\sum_{j=0}^N(\partial_k g_j\cdot D_I x_j+g_j\cdot \partial_k D_I x_j)\]
so that
\[\sum_{j=0}^N \partial_k g_j D_I x_j=-\sum_{j=0}^N g_j\cdot 
\partial_k D_I x_j=-(i_k+1)\sum_{j=0}^N
g_j\cdot D_{I_k} x_j\]
where we have set  $I_k:=(i_1,...,i_{k-1},i_k+1,i_{k+1},...,i_r)$ if $I=(i_1,...,i_r)$.

 From (\ref{(*)}) it therefore follows that 
\begin{equation}\label{(**)}
 \sum_{j=0}^N \partial_k 
g_j\cdot D_I x_j=0 \hbox{ \rm for } |I|\leq {m-2}, 
\end{equation}
and we also get 
\begin{equation}\label{(***)}
\sum_{j=0}^N \partial_k g_j\cdot D_I
x_j =  -(i_k+1)\sum_{j=0}^N g_j\cdot D_{I_k} x_j \hbox{ \rm if }
 |I|= {m-1}. \end{equation}
Because of (\ref{(*)}), $F(a^{m-1})$ factors as stated, and we shall use 
(\ref{(**)}) to compare the induced
map $\overline \phi_m$ with $\phi_m$.

If $g\in (\cK_{m-1})_x$ is as above, $\phi_m$ is given locally by  
\begin{align*}
\phi_m(g)&=a^m(\sum_{j=0}^N g_j\otimes X_j)
=\sum_{j=0}^N g_j \sum_{|I|\leq m} D_I x_j du^I \\
&=\sum_{j=0}^N g_j \sum_{|I|=m} D_I x_j du^I 
=\sum_{|I|=m} (\sum_{j=0}^N g_j D_I x_j)du^I.
\end{align*}
The map $\overline \phi_m$ is given by 
\begin{align*}
\overline \phi_m(g)&= \overline
\phi_m(\sum_{j=0}^N g_j\otimes x_j) 
=\sum_{j=0}^N dg_j\otimes a_x^{m-1}(1\otimes x_j) \\
&=\sum_{j=0}^N \sum_{k=1}^r \partial_k g_j \sum_{|I|\leq {m-1}} D_I 
x_j du_k du^I 
=\sum_{|I|=m-1}
\sum_{k=1}^r \sum_{j=0}^N \partial_k g_j \cdot D_I x_j du_k du^I \\
&=-\sum_{|I|=m-1} \sum_{j=0}^N g_j \sum_{k=1}^r (i_k+1) D_{I_k} x_j 
du_k du^I \\
&=-\sum_{j=0}^N g_j \sum_{|I|=m-1}\sum_{k=1}^r (i_k+1)D_{I_k} du_k 
du^I\\
&=-\sum_{j=0}^N g_j \sum_{|J|=m}(\sum_{k=1}^r j_k)D_J du^J
=-m\sum_{|J|=m}
\sum_{j=0}^N g_j \cdot D_J x_j du^J,
\end{align*}
where we used (\ref{(**)}) and (\ref{(***)}).
This completes the proof of the theorem.
\end{proof}

Given $f \colon X\to \IP(V)$ as before. Recall that the $m$th order osculating space of
$X$ at a point $x\in X$  is defined to be the subspace $\Osc^m_X(x):=\IP(\Image a^m(x)) \subset \mathbb P(V)$.
We 
let $s(m)$ denote the
dimension of $\Osc^m_X(x)$ for a general point $x$, i.e., the map
\[a^m \colon V_X\to \P m\]
 has generic rank $s(m)+1$. Let $U_m\subseteq X$ be
the open dense where $a^m$ has this rank. 
Set ${\cP}_{m} := { \Image} (a^{m})$. On $U_{m}$ the sheaf ${\cP}_{m}$ is
a $(s(m)+1)$-bundle, the $m$th \emph{osculating bundle} of $X$.
Hence there is a rational map, which
is a morphism on $U_m$, 
\[\psi_m \colon X\dashrightarrow \Grass_{s(m)+1}(V),\]
called the
$m$th Gauss map, or $m$th associated map, of $X$ (see \cite{P83}*{p.~336}).

We saw in Proposition \ref{fund} that the second fundamental
form of the kernel $\cK_{m-1}$ of
\[a^{m-1} \colon V_X\to {\cP}_{m-1}\]
 restricted to $U_{m-1}$, induces 
 \[ d\psi_{m-1} \colon \cK_{m-1}\otimes {\cP}_{m-1}^\vee\to \Omega^1_X,\]
and hence a map (on $U_{m-1}$) 
\[\cK_{m-1}\to \Omega^1_X
\otimes {\cP}_{m-1}\hookrightarrow \Omega^1_X\otimes \P {m-1}.\]
It follows from Theorem \ref{fact} that this  map factors through $\Omega^1_X\otimes
S^{m-1}\Omega^1_X\otimes {\cL}$, and hence we get an induced map 
\[\overline \phi_m \colon  \cK_{m-1}\to S^m\Omega^1_X\otimes {\cL}.\]
The equality
\[\overline\phi_m=-m\phi_m.\] 
can be interpreted as a verification of the statement ``the
$m$th fundamental form of $X$ is equal to the differential of the $(m-1)$th
Gauss map'' (cf. \cite{GH}*{(1.62), p.~379}, 
\cite{L}*{Remark, p.~307}, \cite{DeDiI}*{Thm.~1.18, p.~1202}, and \cite{EN}*{Thm.~3.3}). 

\begin{remark}
In \cite{GH}*{(1.52), p.~376} and in the introduction of \cite{EN}, it is noted that the vanishing of the third fundamental form implies that $X$ is
contained in its second order osculating space at a general point of $X$. In fact, as stated in \cite{EN}*{Remark~2.14}, this statement generalizes to higher order fundamental forms. A simple way to see this is as follows.

Consider the $m$th fundamental form $\Phi_m \colon \cK_{m-1}/\cK_m \to S^m\Omega^1_X\otimes \cL$.
Assume $\Phi_m$ is zero on $U_m$, where $U_m\subseteq X$ denotes the open on which $\rank a^m$ is constant. Since $\Phi_m$ is (generically) injective, it follows that $\cK_{m-1}=\cK_m$ on $U_m$. Hence we get an equality of osculating bundles $\cP_m|_{U_m}=\cP_{m-1}|_{U_m}$. This means that ``adding derivatives'' does not make the $m$th osculating  spaces bigger than the $(m-1)$th. But this can only happen if the $(m-1)$th osculating spaces are constant, which implies that $X$ is contained in $\Osc^{m-1}_X(x)$, for (any) $x\in U_m$.
\end{remark}

\section{Geometric interpretation}\label{geo}
Now we turn to the geometric interpretation of the fundamental forms as
linear systems on the projectivized tangent spaces of $X$.

Let $x\in X$ be a point such that  dim $\Osc^m_X(x)=s(m)$ and consider
the linear subspace 
\[\cK_{m-1}(x)\subseteq V= H^0(\IP(V),{\cO}_{\IP(V)}(1))\]
whose elements correspond to hyperplanes containing $\Osc^{m-1}_X(x)$.
At the point $x$, the map $a^{m-1}$ is Taylor series expansion:
\[a^{m-1}(x) \colon V\to {\P {m-1}}(x) \cong {\cO}_{X,x}/\mathfrak{m}^m_x.\]
Hence, if $h\in V$, then $h\in \cK_{m-1}(x)$ iff $a^{m-1}(x)(h)\in
\mathfrak{m}^m_x$. In particular, this shows that the hyperplane $H$ defined 
by $h=0$ is
such that $x\in X\cap H$ is a point of multiplicity $\geq m$, and that this
multiplicity is $> m$ iff $H\supseteq \Osc^m_X(x)$.

Consider the induced map
\begin{align*} \Phi_m(x) \colon \cK_{m-1}(x)/\cK_m(x)\to &(S^m\Omega^1_X\otimes 
{\cL})(x)\cong
\mathfrak{m}^m_x/\mathfrak{m}^{m+1}_x \\ 
&\cong H^0(PT(x), {\cO}_{PT(x)}(m)),
\end{align*}
where we have set
$PT(x):={\IP}(\mathfrak{m}_x/\mathfrak{m}^2_x)$, the projectivized tangent space
to $X$ at $x$.
Thus $\Phi_m(x)$, the $m$th fundamental form of $X$ at $x$, gives 
rise to 
a linear system of degree $m$ and dimension $s(m)-s(m-1)-1$ on
$PT(x)\cong {\IP}^{r-1}$.
Denote this linear system by $|\Phi_m(x)|$.
The geometric interpretation of the members of this linear system is 
as follows:
Let $H$ be the hyperplane defined by $h\in \cK_{m-1}(x)$. Then there 
is an exact
sequence of ${\cO}_X$-modules 
\[0\to (h)\to \mathfrak{m}_{X,x} \to \mathfrak{m}_{X\cap H,x}\to 0\]
and a surjection
\[\mathfrak{m}_{X,x}^m/\mathfrak{m}_{X,x}^{m+1}\to \mathfrak{m}_{X\cap 
H,x}^m/\mathfrak{m}_{X\cap H,x}^{m+1}.\]
 Then the inclusion 
\begin{align*}
\Proj(\oplus_{m\geq 0}
\mathfrak{m}_{X\cap H,x}^m/\mathfrak{m}_{X\cap H,x}^{m+1}) &\hookrightarrow
\Proj(\oplus_{m\geq 0}\mathfrak{m}_{X,x}^m/\mathfrak{m}_{X,x}^{m+1})\\ 
&={\IP}(\mathfrak{m}_x/\mathfrak{m}^2_x) =PT(x)
\end{align*}
is given as the zeroes of
$\Phi_m(x)(h)\in H^0(PT(x),{\cO}_{PT(x)}(m))$.
\medskip

We have thus shown the following:

\begin{proposition}\label{3} Let $U_m\subseteq X$
be the open dense such that dim $\Osc^i_X(x)=s(i)$ for $i=m-1,m$. Then
the $m$th fundamental form 
\[\Phi_m \colon cK_{m-1}/\cK_m\to S^m\Omega^1_X\otimes {\cL}\]
gives a family (over $U_m$) of linear
systems $|\Phi_m(x)|$ on $PT(x)\cong {\IP}^{r-1}$, of degree $m$ and
dimension $s(m)-s(m-1)-1$. The members of $|\Phi_m(x)|$ are the projectivized
tangent cones of $X\cap H$ at $x$, for $H\in {\IP}(V^\vee)$, $H\supseteq
\Osc^{m-1}_X(x)$, $H\not\supseteq \Osc^m_X(x)$.
\end{proposition}
\medskip

The \emph{Jacobian} of a linear
system of divisors on a projective space  is the linear system generated by
the partial derivatives of the members of the original system. 
More generally, let $X$ be a variety and assume
$\cK\subseteq S^m\Omega^1_X$ is a subsheaf. Define the \emph{Jacobian} $J(\cK)$ of $\cK$ 
to be the image in $S^{m-1}\Omega^1_X$ of $\cK \otimes (\Omega^1_X)^\vee$ 
under the natural contraction map 
\[S^m\Omega^1_X \otimes (\Omega^1_X)^\vee\to 
S^{m-1}\Omega^1_X,\]
given locally by sending $v_1\cdots v_m\otimes w^\vee$ to 
$\sum_{i=1}^m w^\vee (v_i)v_1\cdots \widehat{v_i}\cdots v_m$ \cite{FH}*{(B.14), p.~476}.
The next theorem says that the Jacobian of the linear systems associated to the
$m$th fundamental form is contained in the linear systems associated 
to the $(m-1)$th  fundamental form. This result was stated in \cite{GH}*{(1.47), p.~373}, where an
analytic proof was sketched in the case $m=3$. See also \cite{Lbook}*{4.2}, \cite{DeDiI}*{Thm.~1.12, p.~1199}, \cite{DeI}*{Cor.~3.5, p.~5143}, and \cite{EN}*{Remark~2.14}.

\begin{theorem}\label{jac}
The Jacobian of the linear system $|\Phi_m(x)|$ is contained in the linear system $|\Phi_{m-1}(x)|$.
\end{theorem}

\begin{proof}
The theorem follows from the next proposition.
\end{proof}

\begin{proposition}\label{4}
Set $\cK:=\Phi_m(\cK_{m-1}/\cK_{m})\otimes \cL^{-1}\subseteq S^m\Omega^1_X$.
Then 
\[J(\mathcal K)\subseteq \Phi_{m-1}(\cK_{m-2}/\cK_{m-1})\otimes \cL^{-1}\subseteq S^{m-1}\Omega^1_X.\]
\end{proposition}

\begin{proof} By restricting to an open subset of $X$, we may assume that
the ranks of $a^m$ and $a^{m-1}$ are constant.  It suffices to show (locally) that $J(\mathcal K) \subseteq \phi_{m-1}(\mathcal K_{m-2})\otimes \mathcal L^{-1}$.
We use the local description of $\phi_m$ given in the proof of Theorem \ref{fact}. Let $g\in (\mathcal K_{m-1})_x$. Then $\phi_m(g)=\sum_{|I|=m} (\sum_{j=0}^N g_j D_I x_j)du^I$. The contraction map sends
$\phi_m(g)\otimes \partial U_k$ to $\sum_{|I|=m-1} (\sum_{j=0}^N g_j D_{I^k} x_j)du^{I^k}$, where $I^k:=(i_1,\dots, i_k-1,\dots,i_r)$. Now $g\in (\mathcal K_{m-2})_x$, since
$\mathcal K_{m-1}\subseteq \mathcal K_{m-2}$, hence $\sum_{|I^k|=m-1} (\sum_{j=0}^N g_j D_{I^k} x_j)du^{I^k}\in \phi_{m-1}(\mathcal K_{m-2})_x$.
\end{proof}

\bigskip

The geometrical interpretation of the fundamental forms gives of
course a finer invariant for the osculating behavior of $X$ than just
the dimensions $s(m)$ of the osculating spaces.  The simplest way to
illustrate this, is to look at the case of surfaces with $s(2)=4$. 
\medskip

\begin{example}\label{Togliatti} (Togliatti's Del Pezzo surface \cite{T}*{pp.~261}, \cite{S}*{Ex.~1, p.~248}.)  This is the (toric)
surface  $X\subset \IP^5$ given by the rational parameterization $f \colon  \IP^2 \dashrightarrow \IP^5$ where
\[f(x,y)= (1:x:y:xy^2:x^2y:x^2y^2).\] 
It is the projection of the Del
Pezzo surface of degree 6 in $\IP^6$ from the common point of its
second order osculating spaces (the construction can be generalized to
higher degrees).  In this case, the second fundamental forms are
pencils with no base point. 

To see this, we first find (local) equations for $X$:
$G_1:=X_3-X_1X_2^2=0$, $G_2:=X_4-X_1^2X_2=0$, $G_3:=X_5-X_1^2X_2^2=0$. We take $u_1=x$ and $u_2=y$.
The map $\phi_2 \colon \mathcal K_1\to S^2\Omega^1_X\otimes \mathcal L$ is then given locally by the matrix product $\overline A^{(2)} \cdot B_1$, where $\overline A^{(2)}$ is the matrix obtained by taking the last three rows of the matrix of the map $a^2$:
\[A^{(2)}=\left(\begin{array}{cccccc}
1&x&y&xy^2&x^2y&x^2y^2\\
0&1&0&y^2&2xy&2xy^2\\
0&0&1&2xy&x^2&2x^2y\\
0&0&0&0&y&y^2\\
0&0&0&2y&2x&4xy\\
0&0&0&x&0&x^2
\end{array}
\right)
\]
and the matrix
\[K_1=\left(\begin{array}{ccc}
2xy^2 & 2x^2y & 3x^2y^2\\
-y^2&-2xy&-2xy^2\\
-2xy & -x^2 & -2x^2y\\
1 & 0 & 0\\
0&1&0\\
0&0&1
\end{array}
\right)
\]
is obtained from the matrix
\[
\left(\begin{array}{ccc}
\partial G_1/\partial X_0&\partial G_2/\partial X_0 & \partial G_3/\partial X_0\\
\partial G_1/\partial X_1&\partial G_2/\partial X_1 & \partial G_3/\partial X_1\\
\vdots & \vdots & \vdots\\
\partial G_1/\partial X_5&\partial G_2/\partial X_5 & \partial G_3/\partial X_5
\end{array}
\right)
\]
by substituting $X_0=1$,  $X_1=x$, $X_2=y$, $X_3=xy^2$, $X_4=x^2y$, $X_5=x^2y^2$.
(Alternatively, one can compute directly the kernel of the matrix $A^{(2)}$.)
Hence the map $\phi_2$ is given by the matrix
\[\overline A^{(2)} \cdot K_1=\left(\begin{array}{ccc}
0&y&y^2\\
2y&2x&4xy\\
x&0&x^2
\end{array}
\right)
\]
Thus the image of $\phi_2$ is generated by $2ydxdy+xdy^2$ and $ydx^2+2xdxdy$. This means that the linear system $|\Phi_2(x_0,y_0)|$ at the point $(1:x_0:y_0:x_0y_0^2:x_0^2y_0:x_0^2y_0^2)$ is equal to $\langle 2y_0v_1v_2+x_0v_2^2, y_0v_1^2+2x_0v_1v_2\rangle$, where $(v_1:v_2)$ are coordinates on $PT(f(x_0,y_0))\cong \mathbb P^1$. These linear systems have no base points.

To find the third fundamental forms, note that the kernel of $A^{(2)}$ is given by the column matrix $K_2:=(-x^2y^2,xy^2,x^2y,-x,-y,1)^T$. Let
\[
\overline A^{(3)}=\left(\begin{array}{cccccc}
0&0&0&0&0&0\\
0&0&0&0&1&2y\\
0&0&0&1&0&2x\\
0&0&0&0&0&0
\end{array}
\right)
\]
denote the matrix obtained by taking the last four rows of the matrix $A^{(3)}$.
Then the map $\Phi_2$ is given by the product $\overline A^{(3)} \cdot K_2=(0,y,x,0)^T$.
Therefore we get  the linear system $\Phi_3(x_0,y_0)=\langle y_0v_1^2v_2+x_0v_1v_2^2\rangle$.
We observe that the partial derivatives of the generator for $\Phi_3(x_0,y_0)$ are the generators for $\Phi_2(x_0,y_0)$, i.e., the Jacobian of $\Phi_3(x_0,y_0)$ is equal to $\Phi_2(x_0,y_0)$.
\end{example}

\begin{example} Any ruled, non-developable surface in $\IP^N$, $N\ge 5$, has
$s(2)=4$, hence the linear systems corresponding to the second
fundamental forms are 1-dimensional.  It is easy to see that they have
a base point, corresponding to the (direction of the) ruling (see Proposition \ref{ruled}).  In
\cite{GH}*{p.~377}, the authors ask whether this only occurs for ruled
surfaces.  However, in their Appendix B, they assert that the answer is no. Indeed,
Shifrin \cite{S}*{p.~248} gave an explicit example of a non-ruled surface 
such that the second fundamental forms have a base point.
A slightly modified version of this surface is given by the 
rational parameterization $f \colon \IP^2 \dashrightarrow \IP^5$, where
\[f(x,y)= (1:x+y^2:y:y^3+3xy:y^4+6xy^2+3x^2:y^5+10xy^3+15x^2y),\]
which satisfies a ``differential heat equation''
\[ \partial^2 f/\partial y^2 = \partial f/\partial x.\]
This surface has $s(2)=4$, it is not ruled, but the second fundamental 
forms have a base point, as was also shown in \cite{Te}*{pp.~49--51}.

In fact, any surface of heat equation type has this property \cite{S}*{Thm.~(2.14), p.~237}. Suppose the surface is given by a parameterization $f(x,y)$, satisfying
\[ \partial^2 f/\partial y^2 = \varphi(x,y)\partial f/\partial x,\]
for some function $\varphi(x,y)$. Let $A^{(1)}$ denote the matrix corresponding to the map $a^1$. Then with the notation of the previous example, $A^{(1)}\cdot K_1=0$, in particular $\partial f/\partial x \cdot K_1=0$. The last row of the matrix $\overline A^{(2)}$ is given by $\partial^2 f /\partial y^2$, so that $\partial^2 f /\partial y^2 \cdot K_1=0$.  Hence each linear system is generated by linear forms in $v_1^2$ and $v_1v_2$ and therefore has a base point $(0:1)$.

In Shifrin's example, local equations for the surface $X$ are 
$G_1:=X_3-3X_1X_2+2X_2^3=0$, $G_2:=X_4+ 2X_2^4-3X_1^2=0$, $G_3:=X_5-6X_2^5+20X_1X_2^3-15X_1^2X_2=0$.
Computations as in Example \ref{Togliatti} then give
\[
\overline A^{(2)}\cdot K_1=
\left(\begin{array}{cccccc}
0&3&15y\\
3&12y&30(x+y^2)\\
0&0&0
\end{array}
\right),
\]
so that $\Phi_2(f(x_0,y_0))=\langle v_1^2,v_1v_2 \rangle$. 

 With notations as in the previous example, set
\[
\overline A^{(3)}=\left(\begin{array}{cccccc}
0&0&0&0&0&0\\
0&0&0&0&0&15\\
0&0&0&0&6&30y\\
0&0&0&1&4y&10x+10y^2
\end{array}
\right)
\]
To compute the product $\overline A^{(3)}\cdot K_2$, we only need to know the three last entries in $K_2$. We find that
$K_2=(*,*,*,-10x+10y^2,-5y,1)^T$, and hence we get $\overline A^{(3)}\cdot K_2=(0,15,0,0)^T$.
Therefore $\Phi_3(f(x_0,y_0))=\langle v_1^2v_2 \rangle$, again confirming that the Jacobian of the third fundamental form is equal to the second fundamental form.
\end{example}

\begin{example} There are also non-rational non-ruled surfaces in
$\IP^N$ such that $s(2)=4$. An example is the surface $\Gamma_2$ provided
by Dye \cite{Dye}*{p.~1}:  Consider the intersection of three
quadrics $\sum_{i=0}^5 b_{i}^jX_{i}^2=0$, $j=0,1,2$, in $\IP^5$, where
the $b_i$ are distinct (and nonzero) elements of the base field.  It
was shown in \cite{Te}*{pp.~53--56}, that in this case the second fundamental
forms are pencils without a base point.  
According to \cite{S}*{Thm.~(2.17), p.~239}, this surface must be of ``wave equation type''.
\end{example}

\section{Ruled varieties}

Let $Y$ be a smooth variety of dimension $n$ and $g \colon Y\to \Grass_{e+1}(V)$ a
morphism, where $V$ is a vector space of dimension $N+1$. Let $V_Y\to {\cE}$ denote the pullback via $g$ of
the tautological $(e+1)$-quotient on $\Grass_{e+1}(V)$. Set $X:= {\IP}({\cE})$, let $ f\colon X\to {\IP}(V)$ denote  the induced morphism and  $\pi
\colon X\to Y$ the projection. Set $\cL :=f^*\mathcal O_{\IP(V)}(1)$. Note that $\pi_*\cL=\cE$ \cite{H}*{Ex.~8.4.(a), p.~253}.

Let $x\in \IP (\cE)$ be a general point and $U\subseteq Y$ an open subset, with $\pi(x)\in U$, such that $\pi^{-1}(U)\cong U\times \IP^e$. Then we can find coordinates such that, around $x\in U\times \mathbb A^e$, the morphism $f$ is parameterized by
\[f(u_1,\dots,u_n,t_1,\dots,t_e)=(1:x_1(\underline u, \underline t):\cdots : x_N(\underline u, \underline t)),\]
where the $x_i$ are linear in the $t_j$. Let again $\overline A^{(m)}$ denote the last $\binom{n+e+m-1}{m-1}$ rows of the matrix defining (locally) the homomorphism $a^m_X\colon V_X\to \cP^m_X(\cL)$.
Since the $x_i$ are linear in the $t_j$, their partial derivatives of order $\ge 2$ with respect to the $t_j$ are $0$. The first column is also $0$, so in order to study the $m$th fundamental forms  we can replace the matrix $\overline A^{(m)}$ by the $(\binom{m+n}n+e\binom{m-1+n}n)\times N$-matrix (we use the same name)
\[
\overline A^{(m)}=\left(\begin{array}{cccc}
\partial^m x_1/\partial u_1^m&\partial^m x_2/\partial u_1^m&\ldots&\partial^m x_N/\partial u_1^m\\
\partial^m x_1/\partial u_1^{m-1}\partial u_2&\partial^m x_2/\partial u_1^{m-1}\partial u_2&\ldots&\partial^m x_N/\partial u_1^{m-1}\partial u_2\\
\vdots & \vdots& \vdots & \vdots\\
\partial^m x_1/\partial u_n^m&\partial^m x_2/\partial u_n^m&\ldots&\partial^m x_N/\partial u_n^m\\
\partial^m x_1/\partial u_1^{m-1}\partial t_1&\partial^m x_2/\partial u_1^{m-1}\partial t_1&\ldots&\partial^m x_N/\partial u_1^{m-1}\partial t_1\\
\vdots & \vdots& \vdots & \vdots\\
\partial^m x_1/\partial u_n^{m-1}\partial t_e&\partial^m x_2/\partial u_n^{m-1}\partial t_e&\ldots&\partial^m x_N/\partial u_n^{m-1}\partial t_e
\end{array}
\right).
\]

\medskip

\begin{proposition}\label{ruled}
Let $f \colon X=\mathbb P(\mathcal E) \to \mathbb P(V)$ be a ruled variety as above. Let $x\in X$ be a general point,  let $L_x:=\pi^{-1}(\pi(x))$ denote the ruling containing $x$, and let $PT_L(x)\subset PT(x)$ denote the projectivized tangent space to $L_x$ at $x$. Then each member of the linear system $|\Phi_m(x)|$, for $m\ge 2$, contains 
$PT_L(x)$. In particular, if $n=1$, $PT_L(x)$ is a fixed component of each member, and if $n\ge 2$ and $m\ge 3$, then each member is singular along $PT_L(x)$.
\end{proposition}

\begin{proof}
Let $v_1,\dots,v_n,w_1,\dots,w_e$ denote homogeneous coordinates on the tangent space $PT(x)\cong \mathbb P^{n+e-1}$, where $v_i$ corresponds to $du_i$ and $w_j$ to $dt_j$. The subspace $PT_L(x)\subset PT(x)$ is defined by $v_1=\cdots =v_n=0$. It follows from the shape of
$\overline A^{(m)}$  that   $|\Phi_m(x)|$ consists of hypersurfaces defined by some linear combination of the monomials 
\[v_1^m, v_1^{m-1}v_2, \dots, v_n^m,v_1^{m-1}w_1,\dots,v_1^{m-1}w_e,v_1^{m-2}v_2w_1,\dots, v_n^{m-1}w_e.\]
The first statement follows from this. So  does the third, by taking partial derivatives. If $n=1$, then $PT_L(x)$ has dimension $n+e-1-1=e-1$, hence is a hyperplane in $PT(x)$.
\end{proof}

\begin{corollary}
\[\dim |\Phi_m(x)| \le \binom{n+m-1}{m}+e\binom{n+m-2}{m-1}-1.\]
\end{corollary}

\medskip

One can ask for a ``converse'' statement to Proposition \ref{ruled}, namely how can one characterize ruled varieties given their fundamental forms.
Here we shall just consider the case of surfaces. To show that a surface in $\IP^N$, $N\ge 5$, is ruled, it is necessary that the second fundamental forms are pencils with a base point. However, we have seen that this is not sufficient. 

To show that a surface $X\subset \IP^N$ is ruled, it suffices to show that a general projection $\overline X\subset \IP^3$ is ruled. So we may assume $X\subset \IP^3$. Consider a general point $x\in X$. We may choose coordinates such that $x=(1:0:0:0)$ and $X\cap \mathbb A^3$ has Monge form $(x_1,x_2,f(x_1,x_2))$, where $f(x_1,x_2)= f_2(x_1,x_2)+f_3(x_1,x_2)+f_4(x_1,x_2)+ \cdots$, with $f_2(x_1,x_2)= x_1x_2$. Then the tangent plane to $X$ at $x$ is the plane $x_3=0$, and  $x_1=0$ and $x_2=0$ are the principal tangents. Write $f_3(x_1,x_2)=ax_1^3+bx_1^2x_2+cx_1x_2^2+dx_2^3$. 
The intersection of the zero loci $f_2=0$ and $f_3=0$ in $PT(x)$ is given by $x_1x_2=ax_1^3+dx_2^3=0$, and hence is empty unless $a=0$ or $d=0$.
Say $a=0$, then the principal tangent $x_2=x_3=0$ intersects the surface $X$ in the scheme $k[x_1,x_2,x_3]/(x_3-f,x_2,x_3)=k[x_1]/f(x_1,0)$. Now 
$f(x_1,0)=f_4(x_1,0)+\cdots$, so that the tangent line intersects the surface with multiplicity at least $4$. By \cite{LandsbergLinear}*{Thm.~1, p.~55}, it follows that the tangent is contained in $X$. So if this happens at (almost) all points, then $X$ is ruled.  The form $f_2$ is the second fundamental form at $x$. The form $f_3$ is called the \emph{Fubini cubic form} in \cite{GH}*{pp.~448--449} and it is studied and generalized by Ivey and Landsberg in \cite{IveyL}*{pp.~356--357}. 
It would be interesting to define this cubic form in terms of bundle maps and diagrams as we have done with the second fundamental form. 

We have shown that a sufficient condition for a surface in $\IP^3$ to be ruled is that the intersection of the second fundamental form and the Fubini cubic in the projectivized tangent space $PT(x)$ is non-empty, for almost all points $x\in X$. See also the discussion in \cite{GH}*{pp.~448--449} and in \cite{S}*{pp.~235--236}.

\medskip

Let $\mathcal P^m_Y({\mathcal E})$ denote the sheaf of principal parts of order $m$ of ${\mathcal
E}$, which is a bundle of rank $\binom{n+m}m(e+1)$.  Set $a^m_Y(\cE) \colon V_Y\to \mathcal P^m_Y(\mathcal E)$ equal to the natural homomorphism obtained by composing 
\[V_Y=H^0(\mathbb P(V), \mathcal O_{\mathbb P(V)}(1))_Y\to H^0(X,\mathcal L)_Y\]
 with $H^0(X,\mathcal L)_Y =H^0(Y,\pi_*\mathcal L)_Y=H^0(Y,\mathcal E)_Y\to \mathcal P^m_Y(\mathcal E)$ \cite{numchar}*{§ 6, p.~492}. Note that 
$a^m_X\colon V_X \to  \mathcal P^m_X(\mathcal L)$ is the composition  $\pi^*V_Y\to \pi^*\mathcal E=\pi^*\pi_*\mathcal L\to \mathcal P^m_X(\mathcal L)$. 

\begin{proposition}\label{pushdown}
For each integer $m\geq 1$, we have a natural isomorphism of exact sequences
\[
\CD
0	@>>>	{S^m\Omega ^1_Y\otimes {\cE}}	@>>>	{\cP^m_Y(\cE)}	@>>>	{\cP^{m-1}_Y({\cE})} @>>> 0	 \\ @.	@V{\alpha _m}V\simeq V	@V{\beta_m}V\simeq V
@V{\beta_{m-1}}V\simeq V @. \\ 
0 @>>> \pi_*({S^m\Omega ^1_X\otimes \cL)}	@>>>
\pi_*\cP^m_X(\cL)	@>>> \pi_*\cP^{m-1}_X(\cL) @>>> 0 \\ 
\endCD \]
and the diagram
\[
\CD
V_Y @>a_Y^m(\cE)>> \cP^m_Y(\cE) \\ @| @V{\beta_m}V\simeq V \\ V_Y 
@>{\pi_*a^m_X}>> \pi_*\cP^m_X(\cL) \\ 
\endCD \]
commutes.
\end{proposition}

\begin{proof} Define $\beta_m$ as the adjoint of the composition of the
natural maps 
\[
\pi^*\cP^m_Y(\cE)\to \cP^m_X(\pi^*\cE)=\cP ^m_X(\pi^*\pi_*\cL)\to \cP^m_X(\cL).\]
Thus the map $\alpha_m$ making the diagram commute is defined as the adjoint
of 
\[\pi^*(S^m\Omega^1_Y\otimes{\cE})=S^m\pi^*\Omega^1_Y\otimes 
\pi^*\pi_*\cL \to S^m\Omega^1_X\otimes \cL.
\]
We want to show that the $\alpha_m$'s are isomorphisms.
\medskip

Consider the exact sequence 
\[ 0\to \pi^*\Omega^1_Y\to
\Omega^1_X\to \Omega^1_{X/Y}\to 0,\] 
which gives
\[0\to S^m\pi^*\Omega^1_Y\to
S^m\Omega^1_X\to G_m:=S^m\Omega^1_X/S^m\pi^*\Omega^1_Y\to 0.\]
Then $S^m\Omega^1_X$ has a filtration \cite{H}*{II, Ex.~5.16 (c)}
\[S^m\Omega^1_X = F^0\supseteq F^1\supseteq ...\supseteq F^m\supseteq
F^{m+1}=0\]
 such that 
 \[F^j/F^{j+1}\cong S^j\pi^*\Omega^1_Y\otimes
S^{m-j}\Omega^1_{X/Y}.\] 
Now consider the exact sequences
\[0\to S^m\pi^*\Omega^1_Y\otimes \cL\to
S^m\Omega^1_X\otimes \cL\to G_m\otimes \cL
\to 0\]
and the diagram
\[ \CD
0 @>>> \pi_*(S^m\pi^*\Omega^1_Y\otimes \cL) @>>>    \pi_*
(S^m\Omega^1_X\otimes \cL) \\
@. @| @| @. \\ 
@. S^m\Omega^1_Y\otimes {\cE} @>\alpha_m>> 
\pi_* (S^m\Omega^1_X\otimes \cL) @. \\ 
@. @| @| @. \\ 
@. \pi_*(F^m\otimes \cL) @>>> \pi_*(F^0\otimes \cL) @. \\
\endCD \]
\medskip

In order to show that $\alpha_m$ is an isomorphism, it suffices to show that
the maps 
\[\pi_*(F^{j+1}\otimes \cL)\to \pi_*(F^j\otimes 
\cL)\]
are isomorphisms for $j=0,1,...,m-1$. But this will follow if we can show that
\[\pi_*(F^j/F^{j+1}\otimes \cL)=0\]
for $j=0,1,...,m-1$. Now we have
\begin{align*}
\pi_*(F^j/F^{j+1}\otimes \cL) 
&=\pi_*(S^j\pi^*\Omega^1_Y\otimes S^{m-j}\Omega^1_{X/Y}\otimes \cL)\\
&= S^j\Omega^1_Y\otimes \pi_*(S^{m-j}\Omega^1_{X/Y}\otimes \cL).
\end{align*}
It therefore suffices to show 
\[\pi_*(S^{m-j}\Omega^1_{X/Y}\otimes \cL) =0,\]
 for $j=0,1,...,m-1$.
 Consider the base change map
\[\pi_*(S^{m-j}\Omega^1_{X/Y}\otimes \cL)\otimes k(y) \to
H^0(\pi^{-1}(y),S^{m-j}\Omega^1_{\pi^{-1}(y)}\otimes \cL_y).\]
Since $\pi \colon X\to Y$ is a projective fiber bundle, with $\pi^{-1}(y)\cong \mathbb P^e$, the right hand side has constant dimension for $y\in Y$.
By Grauert's theorem \cite{H}*{Cor.~12.9, p.~288}, it follows that $\pi_*(S^{m-j}\Omega^1_{X/Y}\otimes \cL) $ is locally free and that the base change map is an isomorphism. Hence it suffices to show that
$H^0(\IP^e,S^i\Omega^1_{\IP^e}\otimes \cO_{\IP^e}(1))=0$ for $i=1,\ldots, m$. But this holds by Bott's theorem \cite{Bo}*{Prop.~14.4, p.~246}.
\end{proof}

\medskip

We can also define fundamental forms for varieties in Grassmann varieties.  Set $\cK_{m}^\cE:=\Ker a^m_Y(\cE)$. We get a map
$\phi_m(\cE)\colon \cK_{m-1}^\cE \to S^m\Omega^1_Y\otimes \cE$, which induces
\[\Phi_m(\cE)\colon  \cK_{m-1}^\cE/ \cK_{m}^\cE \to S^m\Omega^1_Y\otimes \cE .\]

\begin{corollary} With notations as above, we have
\[\pi_*\Phi_m=\Phi_m(\cE).\]
\end{corollary}

\begin{proof}
This is an immediate consequence of Proposition \ref{pushdown}.
\end{proof}

\medskip

\begin{example} (Rational normal scrolls \cite{PS}.)
Assume $Y=\IP^1$ and $\cE=\bigoplus_{i=0}^e \cO_{\IP^1}(d_i)$, with $0<d_0 \le \cdots \le d_e$. Then $X=\IP(\cE)\subset \IP(V)$, where $V=\bigoplus_{i=0}^e H^0(\IP^1,\cO_{\IP^1}(d_i))$ has dimension $\sum_{i=0}^e(d_i+1)$, is a rational normal scroll
of degree $d:=\sum_{i=0}^e d_i$.

Assume $2\le m\le d_0$. We have $\rank a_X^m=m(e+1)+1$.
The rank of $\cP^m_Y(\cE)$ is $\binom{1+m}m(e+1)=(m+1)(e+1)$, and this is also the rank of $a_{\IP^1}^m(\cE)$. Indeed, we have 
\[\cP^m_Y(\cE)=\cP_{\IP^1}^m(\bigoplus_{i=0}^e \cO_{\IP^1}(d_i))=\bigoplus_{i=0}^e \cP_{\IP^1}^m(d_i)=\bigoplus_{i=0}^e \cO_{\IP^1}(d_i-m)^{m+1},\]
and $a_{\IP^1}^m(\cE)=\bigoplus_{i=0}^e a_{\IP^1}^m$, where each $a_{\IP^1}^m$ has rank $m+1$.

We can parameterize an open subset of $X$  ($X$ is a toric variety) by the map $(\C^*)^{1+e} \to \IP(V)$, given by
\[(t,s_1,\dots,s_e)\mapsto (1:t:\cdots :t^{d_0}:s_1:s_1t:\cdots:s_1t^{d_1}:\cdots:s_e:s_et:\cdots :\cdots:s_et^{d_e}).\]
This gives
\[
\overline A^{(m)}=\left(\begin{array}{cccc}
M_{d_0}^m&s_1 M_{d_1}^m&\ldots&s_eM_{d_e}^m\\
0&M_{d_1}^{m-1}&\ldots&0\\
\vdots & \vdots& \vdots & \vdots\\
0&0&\ldots&M_{d_e}^{m-1}\end{array}
\right),
\]
where $M_{d_i}^m$ denotes the $1\times (d_i+1)$-matrix $(0,\dots,0,1,\binom{m+1}mt,\dots,\binom{d_i}m t^{d_i-m})$.
From this one can deduce, with the notations from the proof of Proposition \ref{ruled} that the linear system $|\Phi_m(x)|$ is generated by $v^m,v^{m-1}w_1,\dots,v^{m-1}w_e$, and hence that its dimension is $e$ and its fixed component is given by $v=0$.

It is not quite clear how to give a geometric interpretation of the fundamental forms for a variety in a Grassmann variety. In the case of a rational normal scroll, we have
\[\phi_m(\cE)\colon \cK^\cE_{m-1} \to S^m \Omega^1_{\IP^1}\otimes \cE =\bigoplus_{i=0}^e S^m\Omega^1_{\IP^1}\otimes \cO_{\IP^1}(d_i). \]
Since $a_{\IP^1}^m(\cE)=\bigoplus_{i=0}^e a_{\IP^1}^m$, we can view $\Phi_m(\cE)$ as giving $e+1$ linear systems of degree $m$ in each $PT_Y(y)$, for $y\in Y=\IP^1$. So it means that the linear system $\langle v^m,v^{m-1}w_1,\dots, v^{m-1}w_e\rangle$ corresponds to $\langle v^m\rangle$ in each of $e+1$ copies of $PT_Y(y)$.
\end{example}

\end{document}